\documentclass[12pt,reqno]{amsart}

\usepackage[latin1]{inputenc}
\usepackage{amsmath}
\usepackage{amsfonts}
\usepackage{amssymb}
\usepackage{graphics}
\usepackage{color}

\usepackage{enumerate}
\usepackage{amssymb,amsmath,amsthm,amscd}
\usepackage{latexsym,verbatim,graphicx,amsfonts}

\usepackage{amscd}
\usepackage{amsmath}
\usepackage{amssymb}
\usepackage{amsthm}
\usepackage{latexsym}
\usepackage{verbatim}
\usepackage{hyperref}

\theoremstyle{plain}
\newtheorem{theorem}{Theorem}[section]
\newtheorem{thm}[theorem]{Theorem}

\newtheorem{corollary}[theorem]{Corollary}
\newtheorem{lem}[theorem]{Lemma}

\theoremstyle{definition}

\newtheorem{ex}[theorem]{Example}
\theoremstyle{remark}
\newtheorem{rem}[theorem]{Remark}

\newcommand{\C}{\mathbb{C}}

\newcommand{\T}{\mathbb{T}}
\newcommand{\D}{\mathbb{D}}
\newcommand{\N}{\mathbb{N}}
\DeclareMathOperator*{\dist}{dist}

\newcommand{\s}{\text{sgn}(\sigma)}
\newcommand{\per}{\mathcal{S}}


\begin{document}
\title[Boundary behavior of optimal approximants]{Boundary behavior of optimal polynomial approximants}
\author[B\'en\'eteau]{Catherine B\'en\'eteau}
\address{Department of Mathematics, University of South Florida, 4202 E. Fowler Avenue,
Tampa, Florida 33620-5700, USA.} \email{cbenetea@usf.edu}
\author[Manolaki]{Myrto Manolaki}
\address{School of Mathematics and Statistics, University College Dublin, Belfield, Dublin 4, Ireland.} \email{arhimidis8@yahoo.gr}
\author[Seco]{Daniel Seco}
\address{Universidad Carlos III de Madrid and Instituto de Ciencias Matem\'aticas, Departamento de Matem\'aticas, Avenida de la Universidad 30, 28911 Legan\'es (Madrid), Spain.} \email{dseco@math.uc3m.es}
\date{\today}

\begin{abstract}
In this paper, we provide an efficient method for computing the Taylor coefficients of $1-p_n f$, where $p_n$ denotes
the optimal polynomial approximant of degree $n$ to $1/f$ in a Hilbert space
$H^2_\omega$ of analytic functions over the unit disc $\D$, and $f$
is a polynomial of degree $d$ with $d$ simple zeros. 
As a consequence, we show that in many of the spaces $H^2_\omega$,
the sequence $\{1-p_nf\}_{n\in \N}$ is
uniformly bounded on the closed unit disc and, if $f$ has no zeros inside $\D$, the sequence
 $\{1-p_nf \}$ converges uniformly to 0 on compact subsets of the complement of the zeros of $f$ in $\overline{\D}, $ and we obtain precise estimates on the rate of convergence on compacta.
We also treat the previously unknown case of a single zero with higher multiplicity.
\end{abstract}

\maketitle

\section{Introduction}\label{Intro}

Consider a sequence of weights $\omega=\{\omega_k\}_{k \in \N}$ such that $\omega_0=1$, $\omega_k >0$
and $\omega_k/\omega_{k+1} \rightarrow 1$ as $k \rightarrow \infty$,
and define the weighted Hardy space $H^2_\omega$ as the space of
analytic functions  $f(z)=\sum_{k=0}^{\infty} a_k z^k$ in the unit disc $\D$ with norm defined
by \begin{equation}\label{defH2W} \|f\|^2_\omega= \sum_{k=0}^{\infty} |a_k|^2
\omega_k.\end{equation} For instance, the choices $\omega_k=
1/(k+1), 1, (k+1)$, give, respectively, the Bergman space $A^2$, the
Hardy space $H^2$ and the Dirichlet space $D$. They are reproducing kernel
Hilbert spaces with the corresponding induced inner product  
$$\left<f,g\right>_\omega = \sum_{k=0}^{\infty} a_k \overline{b_k} \, \omega_k,$$ 
where $g(z) = \sum_{k=0}^{\infty}b_kz^k$.
The reproducing kernel $k(z, w)$ at a point $w \in \D$ is given by 
\begin{equation}\label{reproker}
k(z , w) = \sum_{k=0}^{\infty} \frac{\overline{w}^k z^k}{\omega_k}.
\end{equation} 
For more information on
these spaces see \cite{FMS}, and for more on the particular cases of Hardy, Bergman, and 
Dirichlet spaces, see \cite{Dur,
Gar, DuS, HKZ, EFKMR}.  

In this article, we will examine the boundary behavior of certain 
polynomials that indirectly approximate reciprocals of functions in these analytic function spaces.
More specifically, for a function $f \in H^2_\omega$ (not identically $0$), we say that $p_n$ is the \emph{optimal polynomial approximant} to $1/f$ of
degree $n \in \N$  if $p_n$ is the polynomial
of degree less or equal to $n$ minimizing $\|1-p_nf\|_\omega$. Note that since $1/f$ is not necessarily in the space and since the norm in question is not multiplicative, $p_n$ is not the best approximant of 
 $1/f$ in the classical sense. Existence and uniqueness of $p_n$ follows from the fact 
that $p_n f$ is the orthogonal projection of $1$ onto the finite dimensional space $\mathcal{P}_n \cdot f,$ where $\mathcal{P}_n$ is the set of all 
polynomials of degree at most $n.$ 

This definition arises naturally from classical problems in function
theory.  In particular, a function $f$ is called \emph{cyclic} for a given space if the closed linear span of its polynomial multiples is equal to the space. It is then easy to show that cyclic functions in $H^2_\omega$ are  precisely  those functions for which 
$\|1-p_nf \|_{\omega} \rightarrow 0$ as $n \rightarrow \infty.$ Characterizing cyclic functions is known to be, in general, a very difficult problem.  In $H^2$, the cyclic functions are the outer functions, but in the classical Dirichlet space for example, a complete characterization of the cyclic functions is still an open problem, often referred to as the Brown and Shields Conjecture (see \cite{BS}). In \cite[Theorem 6.1]{BKLSS}, the authors found a characterization of cyclic functions for $H^2$ based on the zeros of the optimal polynomial approximants.  The hope is that additional insight into optimal polynomial approximants in other spaces will give tools to tackle open problems related to cyclicity. 

The term optimal polynomial approximant
was introduced in \cite{BCLSS} but these polynomials arose earlier in the context of digital signal processing (see \cite{Ch} and references therein).  These polynomials are also closely connected to reproducing kernels and orthogonal polynomials in
weighted Hardy spaces:  in fact, it is not hard to see that $ (p_n f)(z)$ is the reproducing kernel evaluated at $0$ of the weighted space  $\mathcal{P}_n \cdot f$ (see \cite{BKLSS}).
In \cite{BCLSS}, the authors studied the rate of decay of the norm of $1-p_n f$ for fairly general $f$.
However, only a few explicit examples have been computed, including those for 
 $1-z$ in any space $H^2_\omega$
in \cite{FMS} and for $(1-z)^s$ for $Re(s) > 0$ in $H^2$ in
\cite{blSimanek}.   In this paper, we are interested in exploring \emph{explicit} computations for simpler functions $f$.

The first goal of the present article is to
provide an efficient method via a closed formula to compute the coefficients of $1-p_n f$, and hence those of
$p_n$, whenever $f$ is a polynomial. Since for a function $f$ such that $f(0)=0$, the optimal
approximants are always identically $0$, we assume throughout that $f(0) \neq 0$.  Our method is first described for a function $f$
which is a polynomial with simple zeros. In that case, a special role will be played by the 
Szeg\H{o} kernel, which will appear in Gram matrices throughout the proofs of our results. (Gram matrices, also called Gramians, are matrices whose entries are given as inner products of vectors for some fixed set of vectors.)  When we move on to dealing with functions with a single zero of higher multiplicity, we will see that this universal structure will be replaced by Hilbert matrices.  This phenomenon is interesting in its own right and we will concentrate on this case at the end of our article. It seems likely that the same principles
should also solve the general case of all zeros, but the
formulas become less manageable.

From now on, for $n \in \N$ and $f$ a
polynomial of degree $d$, we use the following notation:
\begin{itemize}
\item[(N1)] $\hat{g}(k)$, the Taylor
coefficient of order $k \in \N$ of an analytic function
$g$ at $0$.
\item[(N2)] $d_{k,n}=\widehat{(1-p_n f)}(k)$, $k\in \N$.
\item[(N3)] $Z=Z(f)$ is the zero set of $f$.
\item[(N4)] $v^t$ is the transpose of a vector $v$.
\end{itemize}

Our first result considers the slightly more general situation of a polynomial $g$ of degree at most $d$ and its projection $P_n(g)$ onto the space $\mathcal{P}_n \cdot f.$  Theorem \ref{main} provides an efficient way of computing the coefficients of $g - P_n(g)$ and the distance from $g$ to $\mathcal{P}_n \cdot f.$  We will then apply the result to the case that $g \equiv 1$ to obtain information about the optimal polynomial approximants of $f$. 
For any integer 
$m \in \N$, we define $k_m(z,w) = \sum_{k=0}^m \frac{\overline{w}^k z^k}{\omega_k},$ which is the partial sum of the reproducing kernel $k(z,w).$ Notice that $k_m(z,w)$ is a reproducing kernel for the subspace of the weighted space $H^2_{\omega}$ generated by polynomials of degree at most $m$. 

\begin{thm}\label{main}
 Let $f$ be a monic polynomial of degree $d$ with simple zeros $z_1, \ldots, z_d$ that
lie in $ \C \backslash \{0\}$.  Let $g$ be a polynomial of degree at most $d$, and for $n \geq \deg(g),$  let $P_n(g)$ be the projection of $g$ onto $\mathcal{P}_n \cdot f$.  Let $g_{Z}:= (g(z_1), \ldots, g(z_d)) \in \C^d,$ and let
 $E=E_{Z,n}:= (e_{l,m})_{l,m=1}^d$ be the matrix whose coefficients are given by 
$e_{l,m}=k_{n+d}(z_l,z_m).$
Let 
$L(z) =L_n(z) :=\left(  k_{n+d}(z,z_1), \ldots, k_{n+d}(z,z_d) \right).$ Then $E$ is invertible and 
\begin{equation}\label{g-Pg}
\left( g - P_n(g) \right)(z) = L(z) \cdot E^{-1} \cdot g_{Z}^t.
\end{equation}
Moreover, 
\begin{equation}\label{dg-pg}
{\dist}^2 (g,\mathcal{P}_n \cdot f) = \overline{g_{Z}} \cdot E^{-1} \cdot g_{Z}^t,
\end{equation}
and if the zeros of $f$ are outside $\D,$ then there exists a positive constant $C(g,d,\omega)$ such that 
\begin{equation}\label{g-pgdecay}
{\dist}^2 (g,\mathcal{P}_n \cdot f) \leq C(g,d,\omega)\cdot\left(\displaystyle{\sum_{k=0}^n \frac{1}{\omega_k}}\right)^{-1}.
\end{equation}
 \end{thm}
By applying Theorem \ref{main} to $ g \equiv 1,$ we obtain the following corollary.

\begin{corollary}\label{main1}  Let $f$ be a monic polynomial of degree $d$ with simple zeros $z_1, \ldots, z_d$ that
lie in $ \C \backslash \{0\}$, $p_n$ the $n$-th optimal approximant to $1/f$ in $H^2_\omega$, 
$d_{k,n} = \widehat{(1-p_n f)}(k),$ $v_0:= (1,...,1) \in \C^d,$ and $E:=E_{Z,n}$ the matrix as in Theorem \ref{main}.  
Then there exists a unique vector
$A_n=(A_{1,n},...,A_{d,n})$ such that for
$k=0,...,n+d$, we have
\begin{equation}\label{main11}
d_{k,n} = \frac{1}{\omega_k} \sum_{i=1}^d A_{i,n}
\overline{z_i}^{k}.
\end{equation}
Moreover $A_n^t = E^{-1} \cdot v_0 $ and 
\[ {\dist}^2 (1,\mathcal{P}_n \cdot f) =  \sum_{i=1}^d A_{i,n} = v_0 E^{-1} v_0^t.\]
In particular, 
$\sum_{i=1}^d A_{i,n} \in [0,1],$ and if $f$ is cyclic, then
$ \sum_{i=1}^d A_{i,n} \rightarrow 0 \mbox{ as } n \rightarrow \infty.$ 
Also, if $Z:=Z(f)
\subset \D$, then
\[{\dist}^2 (1,[f]) = v_0 K_Z^{-1} v_0^t,\]
where $[f]$ denotes the $z-$invariant subspace generated in
$H^2_\omega$ by $f$ and $K_{Z}$ is the matrix with entries $k_{Z,l,m}=
\left<k (\cdot ,z_{m} ), k(\cdot , z_{l} )\right>_\omega$ for $l, m =1,...,d$.
\end{corollary}

\begin{rem}
Note that the restriction that $f$ is monic in Theorem \ref{main} and in Corollary \ref{main1} is insignificant, since if $f$ is not monic, we can use the above theorem to find the optimal polynomial 
approximants to $\hat{f}(d)/ f,$ and then divide them by $\hat{f}(d)$ to get the optimal polynomial approximants to $1/f$.
\end{rem}
\begin{rem}
The key point of Corollary \ref{main1} is that the matrix $E:=E_{Z,n}$ is of fixed size $d \times  d$ and depends only on the zeros $Z$ 
of $f$ and on $n$.  In addition, the only unknowns ($A_{i,n}, i = 1, \ldots, d$) needed in order to find a
closed formula for the Taylor coefficients $d_{k,n}$ of $1-p_nf$ (and hence, the Taylor coefficients of $p_n$) for all $k=0,...,n+d$ are independent of $k.$ 
\end{rem}
\begin{rem}
Note that since the entries of the matrix $E_{Z,n}$ are given by the reproducing kernels of the subspace of $\mathcal{P}_{n+d}$  evaluated at the
zeros of $f$, $E_{Z,n}$ encodes a finite-dimensional version
of the corresponding Gram matrix $K_{Z}$ of reproducing kernels of the whole space $H^2_{\omega}$. 
\end{rem}

\begin{rem}
The distance formulas in Corollary \ref{main1} are similar to those discussed in a slightly different context in Theorem 3.5 in \cite{Remarks}. 
\end{rem}

Theorem \ref{main} is the key to understanding the behavior of the optimal approximants to reciprocals of 
polynomials on the unit circle $\T$, where we concentrate on the case in which $Z(f) \cap \D =
\emptyset$.  One may ask, for instance, \emph{does the norm convergence of $1-p_nf$ to $0$ (for cyclic $f$)
carry over to pointwise convergence on the unit circle?} For example, whenever $f$ has no zeros in the closed disc, it is cyclic in
$H^2_\omega$, and $1-p_nf$ converge towards 0 exponentially
fast (which also implies uniform convergence on the boundary).  Could the same pointwise convergence hold for any function $f$ that is a 
polynomial, even if it has zeros on the unit circle? Of course, one would have to exclude convergence at the zeros $z_i$ of $f$, since there, 
$(1-p_n f )(z_i) = 1$ for all $n.$  A similar type of question (motivated in part by Proposition 3.2 in \cite{BCLSS}) involves the Wiener norm of a function $h$,
defined by $\|h\|_{A(\T)}=\sum_{k \in \N} |\hat{h}_k|$.  One may ask, then, 
\emph{is the Wiener norm of $p_n f$ uniformly bounded in $n$?}  Note that since the Wiener norm is always larger than the $H^\infty$
norm, boundedness of the Wiener norm would imply that the sequence
$1-p_nf$ is uniformly bounded on $\overline{\D}.$

On the other hand, one may wonder whether a completely different phenomenon can occur, namely, \emph{are there functions $f$ for which
$\overline{\{(1-p_nf)(z_0): n \in \N \}} = \C$, for some $z_0 \in \T$?}  Such functions are called \emph{universal} at the point $z_0.$  
In this paper, we examine the first two questions, and answer them both in the affirmative, for polynomials $f$ with distinct zeros, first in the case where the 
zeros are outside the open unit disc and approximants are considered in the Hardy or Bergman spaces, and second in the case when all zeros are on the circle and the weight $\omega_k$ is monotonic.  The third question relating to universality is addressed in  
\cite{BMS2}.

Thus, we prove the following. 

\begin{thm}\label{Wiener}
Let $f$ be a polynomial with simple zeros such that $Z(f) \cap \D = \emptyset,$ and let $p_n$
be the $n$-th optimal approximant to $1/f$ in the Hardy space $H^2$ or the Bergman space $A^2$. Then there exists a constant $C > 0$ such
that for all $n \in \N$,
\[\|1-p_nf\|_{A(\T)} \leq C.\]
\end{thm}

The behavior shown here is
opposite to universality: it is not possible for a polynomial $f$ to
be universal at a point of the boundary. It is natural to ask whether, for $z_0 \in \T$ the set of accumulation points of $\{(1-p_nf)(z_0): n \in \N \}$
is a singleton. The answer to this question is contained in the following theorem. 

\begin{thm}\label{pointwise}
Let $f$ be a polynomial with simple zeros such that $Z(f) \cap \D = \emptyset,$ and let $p_n$
be the $n$-th optimal approximant to $1/f$ in the Hardy space $H^2$ or the Bergman space $A^2$.
Then 
\[1-p_nf \rightarrow 0 \quad as \quad n \rightarrow  \infty, \]
uniformly on compact subsets of $\overline{\D} \backslash Z(f)$.
\end{thm}

In the case of the Hardy space, Fatou's theorem guarantees that for any cyclic function $f$, a subsequence of the optimal approximants will satisfy $1-p_{n_k}f \rightarrow 0$ at almost every point of the boundary (where $f$ is defined in the non-tangential limit sense), but in Theorem \ref{pointwise} we are looking to predict the behavior at a given point. In this sense, $1-p_nf$ already satisfies some form of overconvergence but here we find out a special situation in that overconvergence happens (a) for the whole sequence of optimal approximants, (b) in the pointwise sense, (c) with control on the exceptional set, which happens to be exactly the finite set $Z(f)$ of the roots of $f$, and (d) convergence is uniform on compact subsets of $\overline{\D} \backslash Z(f)$. 

In Section \ref{Sec2}, we will provide the proof of Theorem \ref{main} Then, in Section \ref{Sec3}, we state some technical lemmas
establishing lower bounds of determinants of some key matrices, and assume those lemmas to 
prove the general result about the boundedness of 
the Wiener norm, Theorem \ref{Wiener},
and the pointwise convergence, Theorem
\ref{pointwise}. We restrict there to the case of the Hardy space. In Section \ref{lemmas}, we establish these technical lemmas. Then we discuss, in Section \ref{Secnew5}, how to extend these results to the setting of $A^2$, therefore completely settling Theorem \ref{Wiener} and \ref{pointwise}.  If all the zeros of the polynomial $f$ lie on the unit circle, a different set of estimates gives rise to convergence rates in uniform norm over compact subsets of $\overline{\D} \backslash Z(f)$ for spaces $H^2_{\omega}$ with monotonic weights. This is the result of Theorem \ref{Th2r2}.
We dedicate Section \ref{Sec7} to the study of the function $(z-1)^t$, for $t \in \N$, in a space $H^2_\omega$, and the computations there will involve inversion of Hilbert-type matrices.
We conclude with some further directions of research, considerations about sharpness
and open questions in Section \ref{Sec6}. 
We would like to thank the referees of an earlier version of this article whose comments led to significant improvements both in the presentation and  content of the current paper, and some of the proofs we present here have been simplified by making use of their ideas.

\section{Computation of coefficients of projections}\label{Sec2}

Let us now turn to the proof of Theorem \ref{main}.  

\begin{proof}
We fix a monic polynomial $f$ with simple zeros $z_1, \ldots, z_d$ that
lie in $ \C \backslash \{0\},$ and let $g$ be a polynomial of degree at most $d$.  
Notice that for any integer $n \geq \deg(g),$ the reproducing kernels $k_{n+d}(z,w)$ still have the reproducing property even when $|w|\geq 1$ in the space of polynomials of degree at most $n+d$.  Therefore, as is well-known, (see, e.g, \cite[Lemma 4.2.6 and Lemma 4.2.3]{EFKMR}) since the $z_i$, $i = 1, \ldots, d$ are distinct, the functions   $k_{n+d}(z,z_i)$ are linearly independent and the matrix $E$ is invertible. Moreover, the kernels $k_{n+d}(z, z_i)$ are orthogonal to $\mathcal{P}_n \cdot f,$ since for any polynomial $p$ of degree at most $n$, $\left< p_n f , k_{n+d}(\cdot ,z_i) \right> = (p f)(z_i) = 0.$  Therefore, these kernels form a basis for the subspace of $\mathcal{P}_{n+d}$ that is orthogonal to 
$\mathcal{P}_n \cdot f,$ and so 
\begin{equation}\label{form1}
(g - P_n(g))(z) = L(z) \cdot A_n^t,
\end{equation}
where $L(z) :=\left(  k_{n+d}(z,z_1), \ldots, k_{n+d}(z,z_d) \right)$ and $A_n \in \C^d.$  In particular, for each $i = 1, \ldots, d,$
$(g - P_n(g))(z_i) = L(z_i) \cdot A_n^t.$ Rewriting \eqref{form1} in matrix form gives that 
$g_{Z}^t = E \cdot A_n^t,$ and therefore 
$\left( g - P_n(g) \right)(z) = L(z) \cdot E^{-1} \cdot g_{Z}^t,$ which is \eqref{g-Pg} in Theorem \ref{main}.

Now, writing $A_n:=(A_{1,n}, \ldots, A_{d,n}),$ we see that
\begin{align*}
{\dist}^2 (g,\mathcal{P}_n \cdot f) & = \|g - P_n(g)\|_{H^2_{\omega}}^2 \\
& = \left< \sum_{i=1}^d A_{i,n} k_{n+d}(z,z_i),  \sum_{i=1}^d A_{i,n} k_{n+d}(z,z_i) \right>_{\omega} \\
& =  \sum_{i=1}^d A_{i,n} \left( \sum_{j=1}^d \overline{A_{j,n} \, k_{n+d}(z_i,z_j)} \right) \\
& = \sum_{i=1}^d A_{i,n}  \overline{g(z_i)} \\
& = \overline{g_{Z}} \cdot E^{-1} \cdot g_{Z}^t,
\end{align*}
which proves \eqref{dg-pg} in Theorem \ref{main}.

Finally, for $n \geq \deg{g},$ we have that
\begin{align*}
\|g - P_n(g)\|_{H^2_{\omega}}^2 &  = \inf \left\{ \| g - p f\|^2_{H^2_{\omega}} : p \in \mathcal{P}_n \right\} \\
& \leq  \inf \left\{ \| g( 1 - p f)\|^2_{H^2_{\omega}} : p \in \mathcal{P}_{n-d} \right\} \\
& \leq C_1(g,d,\omega) \,  {\dist}^2 (1,\mathcal{P}_{n-d} \cdot f),
\end{align*}
where $ C_1(g,d,\omega)$ is a constant, since $g$ is a polynomial and hence a multiplier for 
$H^2_{\omega}.$  Now if the zeros of $f$ are all outside $\D,$ then by \cite{FMS}, 
$$ {\dist}^2 (1,\mathcal{P}_{n-d} \cdot f) \leq C(d,\omega)\cdot \left( \sum_{k=0}^n \frac{1}{\omega_k} \right) ^{-1},$$
where $C(d,\omega)$ is a constant, and therefore 
\begin{equation}
{\dist}^2 (g,\mathcal{P}_n \cdot f) \leq C(g,d,\omega)\cdot \left( \sum_{k=0}^n \frac{1}{\omega_k} \right) ^{-1},
\end{equation}
where $C(g,d,\omega)$ is a constant, as desired, and the proof of the theorem is complete. 
\end{proof}

\begin{rem}
Notice that this proof is simply a reformulation of the fact that in this finite-dimensional setting, $P_n(g),$ being the difference between $g$ and the projection of $g$ onto the space generated by the linearly independent vectors $k_{n+d}(z,z_i)$, can be obtained by the following ratio of determinants (see, e.g, \cite[Lemma 4.2.4]{EFKMR}):
\begin{eqnarray*}
P_n(g)= \frac{1}{\det E} \begin{vmatrix}

g(z) & k_{n+d}(z,z_1) & \cdots  & k_{n+d}(z,z_d)\\

\overline{g(z_1)} & k_{n+d}(z_1,z_1) & \cdots & k_{n+d}(z_1,z_d)\\

\vdots & \vdots & \ddots  & \vdots \\

\overline{g(z_d)} & k_{n+d}(z_d,z_1) & \cdots & k_{n+d}(z_d,z_d)

\end{vmatrix}.
\end{eqnarray*}
\end{rem}
 Corollary \ref{main1} easily follows from Theorem \ref{main}:

\begin{proof}
Set $g \equiv 1.$ In that case, 
$P_n(g) = p_n f, $ where $p_n$ is the optimal polynomial approximant to $1/f$ of degree $n$, and thus, for 
$v_0 = (1, 1, \ldots, 1) \in \C^d$ and $A_n^t = E^{-1} \cdot v_0,$ where 
$A_n = (A_{1,n}, \ldots, A_{d,n})$,
\begin{align*}
(1-p_nf)(z) & = L(z) \cdot E^{-1} \cdot v^t \\
& = \sum_{i=1}^d A_{i,n} k_{n+d}(z,z_i) \\
& =   \sum_{i=1}^d A_{i,n} \left( \sum_{k=0}^{n+d} \frac{\overline{z_i}^k z^k}{\omega_k} \right)\\
& = \sum_{k=0}^{n+d} \left( \frac{1}{\omega_k}  \sum_{i=1}^d A_{i,n} \overline{z_i}^k \right) z^k,
\end{align*}
as desired, and the fact that 
\[ {\dist}^2 (1,\mathcal{P}_n \cdot f) =   \sum_{i=1}^d A_{i,n} = v_0 E^{-1} v_0^t\] follows immediately. 
Notice that this distance must be a number in $[0,1]$ since
$0 \in \mathcal{P}_n \cdot f$.  In addition, if $f$ is cyclic, then $\sum_{i=1}^d A_{i,n}= \|1-p_n f\|^2_{\omega} \rightarrow 0.$

The last part of the corollary comes from the fact that if 
$Z(f) \subset \D,$ then, as $n \rightarrow \infty,$ each of the kernels $k_n(z,z_i)$ converge to the reproducing kernel $k(z,z_i)$ for the whole space $H^2_{\omega}$ at the same point $z_i,$ which is now a point in the unit disc. Thus, this 
kernel is an element of the space,
so the Gramian is invertible, and the result follows. 

\end{proof}

Notice that Theorem \ref{main} also allows us to estimate the entries of $E^{-1}$, which will be important when examining the pointwise convergence of $1-p_nf$ and related Wiener norm estimates when $Z \subset \T.$ 

\begin{corollary}\label{estinv}
Let $f$ be as in Theorem \ref{main} and $E=E_{Z,n}$ be the corresponding Gram matrix. If $Z\subset\mathbb{T}$ and $E^{-1} = (u_{i,j})_{i,j=1}^d$, then there exists a constant $C(Z,\omega)>0$ such that
\begin{equation}\label{estinv1}
|u_{i,j}| \leq C(Z,\omega) \cdot \left( \sum_{k=0}^n \frac{1}{\omega_k} \right) ^{-1}.
\end{equation}
\end{corollary}

\begin{proof}
For each $i=1, \ldots, d,$  choose $g_i$ to be the interpolating polynomial of degree $d$ such that 
$g(z_i) = 1$ and $g(z_j) = 0$ for $j \neq i,$ and apply \eqref{dg-pg} and \eqref{g-pgdecay} to get that 
$$ |u_{i,i}| \leq C(g_i,d, \omega) \cdot \left(\sum_{k=0}^n \frac{1}{\omega_k}\right) ^{-1}.$$ Now, since $E$ is a Gram matrix, $E^{-1}$ is also a Gram matrix (see, e.g., \cite{HJ}), and therefore the entries of $E^{-1}$ can be viewed as inner products of some set of linearly independent vectors $w_i$, $i=1, \ldots, d.$ Applying the Cauchy-Schwarz inequality then gives that 
\begin{align*} |u_{i,j}| = | \left< w_i, w_j \right>| &\leq \sqrt{ \left< w_i, w_i \right>\cdot  \left< w_j, w_j \right>}\\
&\leq \sqrt{C(g_i,d, \omega)  C(g_j,d, \omega) }\left(\displaystyle{\sum_{k=0}^n \frac{1}{\omega_k}}\right) ^{-1}.
\end{align*}
Since there are only a finite number of $g_i,$ we can then choose an appropriate constant $C(Z, \omega)$ such that for all $i,j= 1, \ldots, d,$ we have
\begin{equation*}
|u_{i,j}| \leq C(Z,\omega) \cdot \left(\sum_{k=0}^n \frac{1}{\omega_k} \right) ^{-1},
\end{equation*}
as desired.
\end{proof}

We would like to stress that the main point of  Theorem \ref{main} and Corollary \ref{main1} is the constructive and explicit nature of the coefficients. The following example is included in order to show the efficiency of our method rather than the result. In fact, the result was previously obtained in \cite{BKLSS}, but the approach we present here is clearly faster.

\begin{ex}
Let us apply Corollary \ref{main1} to the function $f(z) = z - 1$ in the setting of the Hardy space $H^2.$ In Section \ref{Sec7}, we will apply a modification of these ideas to a new example. Then $d = 1, z_1 = 1,$ and for each $k \in \N,$ $\omega_k = 1.$  The matrix $E$ is simply a scalar, 
$$E = e_{1,1} = \sum_{k=0}^{n+1} \frac{1}{\omega_k} = n+2,$$ and $$A_n = A_{1,n} = \frac{1}{e_{1,1}} =  \frac{1}{n+2}.$$
Thus $d_{k,n} = \frac{1}{\omega_k} A_{1,n} \cdot 1 = \frac{1}{n+2},$ and therefore,
\begin{equation*}
(1-p_n f )(z) = \sum_{k=0}^{n+1} d_{k,n} z^k =  \sum_{k=0}^{n+1} \frac{1}{n+2} z^k =  \frac{1}{n+2} \frac{1-z^{n+2}}{1-z}, 
\end{equation*}
for $z \neq 1.$
Solving for $p_n$ gives 
\begin{equation}\label{pnH2}
p_n(z) = - \frac{z^{n+2} - (n+2)z + n+1}{(n+2)(1-z)^2},
\end{equation}
which is precisely equation (2.3)  in \cite{BKLSS}.  (Note that $p_n$ is indeed a polynomial, as the numerator in \eqref{pnH2} has a zero of order 2 at 1.) Thus Corollary \ref{main1} provides a computational tool for finding optimal polynomial approximants $p_n$ by first computing 
the coefficients of $1-p_n f$ in an efficient way.  This method may thus lead to being able to identify optimal polynomial approximants for more difficult examples. 
\end{ex}

Notice that it is possible to use the formula for $p_n$ in \eqref{pnH2} to examine the convergence behavior of $p_n$ on the unit circle $\T.$ We would now like to extend this idea, and use Corollary \ref{main1} to get information about the boundary behavior of the optimal  polynomial approximants to $1/f$ when $f$ is a polynomial, which we turn to in Section \ref{Sec3}.

\section{Wiener norm and boundary behavior of optimal approximants for $H^2$}\label{Sec3}

We now turn to the proofs of Theorems \ref{Wiener} and \ref{pointwise}.  We will begin by giving a proof of 
the case when the space under consideration is $H^2,$ and in Section \ref{Secnew5}, will discuss how to 
extend the results to $A^2$.

Without loss of generality, as discussed in the Introduction, we may choose the polynomial $f$ to be monic.  We assume $f$ has degree $d$ and simple zeros $z_i$ for $i = 1, \ldots, d,$ that all lie in $\C \backslash \D$. Let $p_n$ be the 
$n$-th optimal approximant to $1/f$ in $H^2$, and we would like to show that $\|1-p_nf\|_{A(\T)} \leq C < \infty.$  
Without loss of generality, we suppose that $1 \leq d_1 \leq d$ is
such that $|z_i|=1$ for $i=1,...,d_1$ and $|z_i|> 1$ otherwise. (Note that if $d_1 = 0,$ then $1/f$ is analytic in the closed disc, and therefore
$1-p_n f$ converges to $0$ uniformly in the closed disc, and the result certainly follows.)
We also order the zeros so that $|z_i| \leq |z_{i+1}|$ for all $i$.  

The difficulty of the proof lies mainly in estimating
various determinants of matrices constructed from these zeros, and we will see that there will be a distinction in the decay in terms of $n$ of the coefficients
$A_{i,n}$ from Corollary \ref{main1} related to zeros $z_i$ that lie on the unit circle $\T$ versus the ones that lie outside $\T$.
We will need the following three lemmas, which we will prove in Section \ref{lemmas}.  The first lemma examines a matrix that is relevant to the zeros that lie outside the unit circle.  

\begin{lem}\label{matrixB}
Let $\zeta_1, \ldots, \zeta_s$ be distinct complex numbers such that $|\zeta_i|> 1$ for all $i.$ Then the matrix 
$B:=\left( b_{l,m} \right)_{l,m=1}^s$ with $b_{l,m}= \frac{1}{\zeta_l \, \overline{\zeta_m} - 1}$ is positive definite, and in particular, 
$\det(B) > 0.$
\end{lem}

The second lemma shows that the determinant of the matrix $E$ in Corollary \ref{main1} can be bounded below by the product of its diagonal terms.

\begin{lem}\label{diagonal}
Let $1 \leq d_1 \leq d$ be integers, and $z_i \in \C$ be such that $|z_i| = 1$ for $ 1 \leq i \leq d_1$ and $|z_i| > 1$ for $d_1 < i \leq d.$   
If $E:= \left( e_{l,m} \right)_{l,m=1}^d$ with $e_{l,m}= \sum_{k=0}^{n+d} z_l^k\,\overline{z_m}^{k},$ then there exists a constant
$\delta > 0,$ independent of $n,$ such that for every $n,$
\begin{equation}\label{DiagEst}
\det(E) \geq \delta \cdot (n+d+1)^{d_1} \cdot \prod_{l=1}^d |z_l|^{2(n+d+1)}. 
\end{equation} 
\end{lem}

The third lemma gives an estimate of the decay of the coefficients $A_{i,n}$ and consequently allows us to estimate the coefficients $d_{k,n}$.

\begin{lem}\label{growthAin}
The coefficients $A_{i,n}$ from Corollary \ref{main1} have the following growth, as $n \rightarrow \infty$:
\begin{equation*}
A_{i,n} = 
\begin{cases}
  O\left( \frac{1}{n+d+1} \right) \text{ for } 1 \leq i \leq d_1\\      
  o\left( \frac{1}{|z_i|^{n+d+1}} \right) \text{ for } d_1 < i \leq d.
\end{cases}
\end{equation*}
Consequently, for each $1 \leq i \leq d_1,$ there exists a constant $C_i$, independent of $n,$ such that 
$$ \sum_{k=0}^{n+d} |A_{i,n} \, \overline{z_i}^k| \leq C_i,$$ while for $d_1 < i \leq d,$ we have
$$ \sum_{k=0}^{n+d} |A_{i,n} \, \overline{z_i}^k| \rightarrow 0 $$ as $n \rightarrow \infty.$
\end{lem}

Assuming these lemmas for the moment, we can now prove Theorem \ref{Wiener} for the Hardy space $H^2$. 

\begin{proof}[Proof of Theorem \ref{Wiener}.]
Recall from Corollary \ref{main1} that $\left( 1-p_n f \right)(z) = \sum_{k=0}^{n+d} d_{k,n} z^k,$ where
$d_{k,n} = \sum_{i=1}^d A_{i,n} \, \overline{z_i}^k,$ and $A_{i,n}$ satisfy the linear equation specified in the theorem. 
Therefore, the Wiener norm can be estimated as follows:
\begin{equation*}
\|1-p_nf\|_{A(\T)}  =  \sum_{k=0}^{n+d} |d_{k,n}|  = \sum_{k=0}^{n+d} \left|  \sum_{i=1}^d A_{i,n} \, \overline{z_i}^k \right| 
 \leq \sum_{i=1}^d \sum_{k=0}^{n+d}    \left|A_{i,n} \, \overline{z_i}^k \right|. 
\end{equation*}
Now use Lemma \ref{growthAin} to estimate the last quantity above, and conclude that 
\begin{equation*}
\|1-p_nf\|_{A(\T)} \leq \sum_{i=1}^{d_1} C_i + o(1) \leq C < \infty,
\end{equation*}
for some positive constant $C$, thus proving the theorem for $H^2.$
\end{proof}

The estimates we obtained for the coefficients $A_{i,n}$ allow us to get even more precise information about the behavior of 
$1-p_nf$ as stated in Theorem \ref{pointwise}, which we now prove, again for the Hardy space $H^2.$

\begin{proof}[Proof of Theorem \ref{pointwise}]
First note that since point evaluations in $H^2_{\omega}$ are bounded and $f$ is cyclic, $1-p_nf$ converges uniformly to $0$ on compact subsets of $\D$. Therefore, it suffices to 
prove the result for compact subsets $ K \subset \overline{\D}  \backslash \{ z_1, \ldots, z_d \}$ that avoid $1/\overline{z_i}$ for $d_1 < i \leq d.$ 
Let $K$ be such a compact set.  Then by Corollary \ref{main1}, for each $z \in K,$
$$(1-p_n f )(z) = \sum_{k=0}^{n+d} \left( \sum_{i=1}^d A_{i,n} \, \overline{z_i}^k \right) z^k = \sum_{i=1}^d A_{i,n} \cdot \frac{1 - (\overline{z_i} \, z)^{n+d+1}}{1 - \overline{z_i}\,z}.$$ 
Now for $ 1 \leq i \leq d_1,$ the term $ \frac{1 - (\overline{z_i} \, z)^{n+d+1}}{1 - \overline{z_i}\,z}$ is uniformly bounded on $K,$ while according to 
Lemma \ref{growthAin}, $A_{i,n}$ goes to $0$ as $n \rightarrow \infty$.  On the other hand, if $d_1 < i \leq d,$ then 
there exists a positive constant $C$ such that
for each $z \in K,$ 
$$  \left| A_{i,n} \cdot \frac{1 - (\overline{z_i} \, z)^{n+d+1}}{1 - \overline{z_i}\,z} \right| \leq C |A_{i,n}| |z_i|^{n+d+1},$$
which again by Lemma \ref{growthAin}, goes to $0$ as $n \rightarrow \infty,$ thus concluding the proof of Theorem \ref{pointwise} for $H^2.$ 
\end{proof}

\section{Proofs of technical lemmas for $H^2$}\label{lemmas}

We now turn to the proofs of the lemmas, starting with Lemma \ref{matrixB}.

\begin{proof}[Proof of Lemma \ref{matrixB}.]
Let $\zeta_1, \ldots, \zeta_s$ be distinct complex numbers such that $|\zeta_i|> 1$ for all $i,$ and $B$ the matrix with entries 
$b_{l,m}= \frac{1}{\zeta_l \, \overline{\zeta_m} - 1}.$ Writing $w_i : = \frac{1}{\zeta_i},$ we have that 
$$ b_{l,m} = \frac{1}{\frac{1}{w_l} \cdot \frac{1}{\overline{w_m}} - 1} = \frac{w_l \, \overline{w_m}}{1 - w_l \, \overline{w_m}} = \frac{1}{1 - w_l \, \overline{w_m}} - 1 = k(w_l,w_m) - 1,$$
where $k(z,w)$ is the Szeg\H{o} kernel, which is reproducing for $H^2$. Now notice that 
 $k(z,w) - 1$ is the reproducing kernel for the subspace of $H^2$ consisting of functions $f \in H^2$ that vanish 
at the origin, since for such $f$, we have 
$$ \left< f, k( \cdot, w) - 1 \right>_\omega = f(w) - f(0) = f(w).$$
Writing $K(z,w):= k(z,w) - 1,$ we conclude that 
$$ b_{l,m} = K(w_l,w_m) = \left< K( \cdot, w_m),  K( \cdot, w_l) \right>_\omega,$$
that is, $B$ is a Gramian. Since $K(z,w)$ is a reproducing kernel and since the points $w_1, \ldots, w_s$ are distinct, 
it is easy to see that the functions $K(z, w_1), \dots, K(z, w_s)$ are linearly independent, and thus, since the Gramian of a set of 
linearly independent vectors is positive definite, 
 the matrix $B$ is positive definite, 
and in particular $\det(B) > 0.$

\end{proof}

We will now use Lemma \ref{matrixB} to prove Lemma \ref{diagonal}. 

\begin{proof}[Proof of Lemma \ref{diagonal}.]

Let $E:= \left( e_{l,m} \right)_{l,m=1}^d$ be as in Lemma \ref{diagonal}. In what follows, we will use the notation $\per$ to denote the set of all permutations of the indices $\{ 1, \ldots, d \},$ $\s$ to denote the parity of a particular permutation $\sigma \in \per,$ and $\text{id}$ to denote the identity permutation.  Then we have, by the definition of determinant,
\begin{equation*}
\det(E) = \sum_{\sigma \in \per} \left[  \s  \prod_{l=1}^d e_{l,\sigma(l)} \right]   = \sum_{\sigma \in \per} \left[ \s \prod_{l=1}^d \left( \sum_{k=0}^{n+d} z_l^k \, \overline{z_{\sigma(l)}}^k \right)\right].
\end{equation*}
Let us decompose this sum depending on the number of indices a given permutation fixes.  Recall that
$1 \leq d_1 \leq d$ is
such that $|z_i|=1$ for $i=1,...,d_1$ and $|z_i|> 1$ otherwise.
 Let $\mathcal{A}$ be the set of all permutations $\sigma \in \per$ such that $\sigma(i) = i$ for every $1 \leq i \leq d_1$, and, for each 
$0 \leq j < d_1$, and let $\mathcal{B}_j$ be the set of all permutations $\sigma \in \per$ that fix exactly $j$ of the indices in the set $\{ 1, \ldots, d_1\}.$  Then 
\begin{align}
\det(E) =  \prod_{l=1}^d \left( \sum_{k=0}^{n+d} |z_l|^{2k} \right) & + \sum_{\sigma \in \mathcal{A} \backslash \{ \text{id} \}} \s  \prod_{l=1}^d  \left( \sum_{k=0}^{n+d} z_l^k \, \overline{z_{\sigma(l)}}^k \right) \label{breakdown1} \\
	& +  \sum_{j=0}^{d_1-1} \sum_{\sigma \in \mathcal{B}_j} \s  \prod_{l=1}^d  \left( \sum_{k=0}^{n+d} z_l^k \, \overline{z_{\sigma(l)}}^k \right).\label{breakdown2}
\end{align}
Now notice that if $l \neq \sigma(l)$, or if $l = \sigma(l) > d_1,$ then $z_l \, \overline{z_{\sigma(l)}} \neq 1,$ and so 
\begin{equation}\label{estimatesum}
\sum_{k=0}^{n+d} z_l^k \, \overline{z_{\sigma(l)}}^k = \frac{1 - \left( z_l \, \overline{z_{\sigma(l)}}\right)^{n+d+1}}{1 -z_l \, \overline{z_{\sigma(l)}}} = 
 \left( z_l \, \overline{z_{\sigma(l)}}\right)^{n+d+1} \cdot C(l, \sigma, n),
\end{equation}
where 
$$ C(l, \sigma, n) = \frac{\frac{1}{\left( z_l \, \overline{z_{\sigma(l)}}\right)^{n+d+1}}-1}{1 -z_l \, \overline{z_{\sigma(l)}}},$$ which is bounded above, and if either $l$ or $\sigma(l)$ is greater than $d_1$, then 
\begin{equation}\label{Cln}
C(l, \sigma, n)  \rightarrow \frac{1}{z_l \, \overline{z_{\sigma(l)}}-1} \quad \text{ as } n \rightarrow \infty.
\end{equation}
On the other hand, if $ l = \sigma(l) \leq d_1,$ then 
$$\sum_{k=0}^{n+d} z_l^k \, \overline{z_{\sigma(l)}}^k = n + d + 1.$$

Therefore, the first summand in \eqref{breakdown1} is equal to
\begin{equation}\label{I}
(n+d+1)^{d_1} \cdot \left( \prod_{l=d_1+1}^d |z_l|^{2(n+d+1)} C(l, \text{id}, n) \right).
\end{equation}
We can also compute the second summand in \eqref{breakdown1} and it is equal to
\begin{equation}\label{II}
\sum_{\sigma \in \mathcal{A}\backslash \{ \text{id} \}} \s  (n+d+1)^{d_1} \cdot \left( \prod_{l=d_1+1}^d \left( z_l \, \overline{z_{\sigma(l)}} \right)^{n+d+1}  C(l, \sigma, n) \right).
\end{equation}
Finally, the summand in \eqref{breakdown2} consists of sums similar to those giving \eqref{II} except involving powers $(n+d+1)^j,$ with $0 \leq j < d_1-1$ and products over some subset of indices $l$. 

Now notice that if $\sigma \in \mathcal{A},$ since $\sigma$ is bijective from $\{ d_1+1, \ldots, d \}$ to itself,
we have
$$ \prod_{l=d_1+1}^d \left( z_l \, \overline{z_{\sigma(l)}} \right)^{n+d+1} = \prod_{l=d_1+1}^d |z_l|^{2(n+d+1)} =  \prod_{l=1}^d |z_l|^{2(n+d+1)}.$$

Therefore, the determinant of $E$ is equal to
\begin{align}
& (n+d+1)^{d_1}  \cdot \left( \prod_{l=1}^d |z_l|^{2(n+d+1)} \right) \cdot \\
 & \left\{ \prod_{l=d_1+1}^{d} C(l, \text{id}, n)  +  \sum_{\sigma \in \mathcal{A}\backslash \{ \text{id} \}} \s \prod_{l=d_1+1}^{d} C(l, \sigma, n) + r(n) \right\}, \label{detE2}
\end{align}
where $r(n)$ denotes the remainder terms.  
Now, note that \[e_{l,m} = \left< k_{n+d} (z_l, \cdot), k_{n+d}(z_m,\cdot)\right>_{\omega},\] which implies that $E$ is a Gram matrix and by Theorem \ref{main}, it is invertible and thus positive definite. We can conclude that $\det(E)>0$ (for all $n$). 
In addition, $r(n) \rightarrow 0$  as $n \rightarrow \infty.$ On the other hand, by \eqref{Cln} and since if $\sigma \in \mathcal{A},$ we can think of $\sigma$ as a permutation of the indices $d_1+1, \ldots, d,$ 
we have that 
$$ \lim_{n \rightarrow \infty} \left( \prod_{l=d_1+1}^{d} C(l, \text{id}, n)  +  \sum_{\sigma \in \mathcal{A}\backslash \{ \text{id} \}} \s \prod_{l=d_1+1}^{d} C(l, \sigma, n) \right) = \det(B),$$ 
where $B:= \left( b_{l,m} \right)_{l,m=d_1+1}^{d}$ is defined by $b_{l,m} = \frac{1}{z_l \, \overline{z_m}-1}.$
 Therefore, by Lemma \ref{matrixB}, $B$ is positive definite, and so $\det(B) > 0.$ 
Thus, by \eqref{detE2} and since $\det(E) > 0$,
the quantity 
$$ \left\{ \prod_{l=d_1+1}^{d} C(l, \text{id}, n)  +  \sum_{\sigma \in \mathcal{A}\backslash \{ \text{id} \}} \s \prod_{l=d_1+1}^{d} C(l, \sigma, n) + r(n) \right\}$$ is strictly positive for all $n$, and converges as $n \rightarrow \infty$ to a positive quantity, and is therefore bounded below by some constant $\delta > 0.$ Hence 
$$ \det(E) \geq \delta \cdot (n+d+1)^{d_1} \cdot \prod_{l=1}^d |z_l|^{2(n+d+1)},$$ as desired.
\end{proof}

Notice that Lemma \ref{diagonal} essentially shows that the size of the determinant of $E$ as $n$ becomes large is comparable to the product of its diagonal terms. 

Let us now prove Lemma \ref{growthAin}.

\begin{proof}[Proof of Lemma \ref{growthAin}.]
Recall from Corollary \ref{main1} that the coefficients $A_{i,n}$ are obtained as the solution to the linear system $E \cdot A_n^t = v_0^t,$ where $v_0:= (1,...,1) \in \C^d.$  Therefore by Cramer's rule, 
if $E^{(i)}:=(e^{(i)}_{l,m})_{l,m =1}^d $ denotes the matrix obtained from $E$ by replacing the $i$-th column of $E$ by $v_0^t$, we have that $A_{i,n} = \frac{\det(E^{(i)})}{\det(E)}.$  Therefore, $\det(E^{(i)})$ is given by 
\begin{equation}\label{Ei}
\sum_{\sigma \in \per} \s  \prod_{l=1}^d e_{l,\sigma(l)}^{(i)}  = \sum_{\sigma \in \per} \s \cdot\prod_{l=1, \sigma(l) \neq i}^d \left( \sum_{k=0}^{n+d} z_l^k \, \overline{z_{\sigma(l)}}^k \right).
\end{equation}
Now if $1 \leq i \leq d_1,$ then arguing as in Lemma \ref{diagonal}, since in all the sums $\sigma (l) \neq i,$ the highest power of $n+d+1$ that can appear in any term of the expression of $\det (E^{(i)})$ is 
$(n+d+1)^{d_1-1}$, multiplied by a product that is bounded above by a constant multiple of $\prod_{l=1}^d |z_l|^{2(n+d+1)}.$  Therefore, there exists a positive constant $C_1$ such that
$$ |\det(E^{(i)})| \leq C_1 \cdot (n+d+1)^{d_1-1} \cdot \prod_{l=1}^d |z_l|^{2(n+d+1)}.$$ Now applying Lemma \ref{diagonal} gives that, for $ 1 \leq i \leq d_1,$ we have 
$ A_{i,n} = O\left( \frac{1}{n+d+1} \right)$ as $n \rightarrow \infty.$ 

On the other hand, suppose now that $d_1 < i \leq d.$  Recall that 
$\mathcal{A}$ is the set of all permutations $\sigma \in \per$ such that $\sigma(j) = j$ for every $1 \leq j \leq d_1.$
Then again, arguing as in Lemma \ref{diagonal}, 
\begin{align}
& \det(E^{(i)})  =  (n+d+1)^{d_1} \cdot \prod_{\stackrel{l=d_1+1}{l \neq i}}^d  \frac{1- |z_l|^{2(n+d+1)}}{1 - |z_l|^2}  \label{Ei1} \\
&+ \sum_{\sigma \in \mathcal{A} \backslash \{ \text{id} \}} \s  (n+d+1)^{d_1} \cdot  \prod_{\stackrel{l=d_1+1}{\sigma(l) \neq i}}^d  \frac{1- (z_l \, \overline{z_{\sigma(l)}})^{n+d+1}}{1 - z_l \, \overline{z_{\sigma(l)}}}  \label{Ei2} \\
&  + R(n), \label{Ei3}
\end{align}
where $R(n)$ denotes the remainder terms.  Now note that in \eqref{Ei1}, the product is missing a term of order $|z_i|^{2(n+d+1)}$.  In \eqref{Ei2}, each product is missing a term of order
$|z_{i^*}|^{n+d+1} \cdot |z_i|^{n+d+1}$, where $i^* := \sigma^{-1}(i) > d_1,$ and hence by  Lemma \ref{diagonal}, after division by $\det(E)$,  has order of decay at most, say, $|z_i z_{d_1+1}|^{-(n+d+1)}.$
Finally, in \eqref{Ei3}, the highest power of $n+d+1$ that appears is $(n+d+1)^{d_1-1},$ and each product is missing at least one term of order  $|z_i|^{n+d+1}$.  Therefore, after division by $\det(E)$
and using Lemma \ref{diagonal}, we can conclude that 
$ A_{i,n} = \frac{\det(E^{(i)})}{\det(E)} $
has order of decay at most 
$$O \left( \frac{1}{|z_i|^{2(n+d+1)}} +  \frac{1}{|z_i|^{n+d+1} \cdot |z_{d_1+1}|^{n+d+1}} + \frac{1}{(n+d+1) \cdot |z_i|^{n+d+1}} \right),$$ and therefore we obtain that
$ A_{i,n} = o\left( \frac{1}{|z_i|^{n+d+1}} \right)$ as $n \rightarrow \infty,$ as desired.  

Using these estimates, it is now easy to see that for $1 \leq i \leq d_1,$ there is a constant $C_i$ such that
$$ \sum_{k=0}^{n+d} |A_{i,n} \, \overline{z_i}^k| \leq \frac{C_i}{n+d+1}  \cdot   \sum_{k=0}^{n+d} |\overline{z_i}^k| = C_i,$$
while if $d_1 < i \leq d,$
$$ \sum_{k=0}^{n+d} |A_{i,n} \, \overline{z_i}^k| = |A_{i,n}| \cdot \frac{1 - |z_i|^{n+d+1}}{1 - |z_i|}  \rightarrow 0 $$ as $n \rightarrow \infty,$ and the proof of Lemma \ref{growthAin} is complete. 

\end{proof}

\section{Wiener norm and boundary behavior of optimal approximants for $A^2$}\label{Secnew5}

In order to prove Theorems \ref{Wiener} and \ref{pointwise} for the Bergman space $A^2$ (i.e., when $\omega_k = \frac{1}{k+1}$), we need good estimates of the 
partial sums of the reproducing kernel $k(z,w)$ when evaluated at points $z$ and $w$ that are on the unit circle or outside the closed unit disc.  With such
estimates,  one can obtain analogous versions of Lemma \ref{diagonal} and Lemma \ref{growthAin}.  More specifically, we have the following.

\begin{lem}\label{BergmanKernel}
Let $1 \leq d_1 \leq d$ be integers, and $z_i \in \C$ be such that $|z_i| = 1$ for $ 1 \leq i \leq d_1$ and $|z_i| > 1$ for $d_1 < i \leq d.$
Let $ \omega_k = \frac{1}{k+1},$ $ 1 \leq l \leq d$, and let $\sigma$ be a permutation of $\{ 1, 2, \ldots, d \}$ such that $z_l \cdot \overline{z_{\sigma(l)}} \neq 1.$ Then 
\begin{equation}\label{sigmaone}
\sum_{k = 0}^{n+d} \frac{z_l^k \overline{z_{\sigma(l)}}^k}{\omega_k} = (n + d + 2) \left( z_l \overline{z_{\sigma(l)}}\right)^{n+d+1}  \cdot C(l,\sigma,n),
\end{equation}
where $ C(l,\sigma,n) \rightarrow  \frac{1}{z_l  \overline{z_{\sigma(l)}}-1}$ as $n \rightarrow \infty.$
\end{lem}

\begin{proof}
Let $\omega_k = \frac{1}{k+1},$ let $|z| \geq 1,$ $ z \neq 1,$ and let $N$ be an integer.  Then the partial sum of the reproducing kernel for the Bergman space equals
\begin{equation*}
\sum_{k=0}^N \frac{z^k}{\omega_k} = \left( \sum_{k=0}^N z^{k+1} \right)' = \left( \frac{z(1-z^{N+1})}{1-z} \right)'.
\end{equation*}
A direct calculation shows that the latter is equal to
\begin{equation*}
\frac{1-z^{N+2}}{(1-z)^2} + \frac{(N+2)z^{N+1}}{z - 1},
\end{equation*}
and thus, 
\begin{equation}\label{BEst}
\sum_{k=0}^N \frac{z^k}{\omega_k} = \frac{(N+2)z^{N+1}}{z - 1} \left[ 1 + o(1) \right].
\end{equation}
Applying \eqref{BEst} to $N = n+d$ and $z = z_l  \overline{z_{\sigma(l)}}$ gives the desired result. 
\end{proof}

\begin{rem}
Whenever we can find an analogue to \eqref{sigmaone} for other spaces $H^2_\omega$, we expect that the limit as $n \rightarrow \infty$ of $C(l , \sigma, n)$ will remain unchanged. This 
seems to indicate that the Szeg\H{o} kernel plays a key role in diverse classes of weighted Hardy spaces. 
\end{rem}

The above estimate allows us to get the following version of Lemma \ref{diagonal}.

\begin{lem}\label{diagonalBergman}
Let $ \omega_k = \frac{1}{k+1},$  and let $d$, $d_1$, and $z_i$ be as in Lemma \ref{BergmanKernel}. 
If $E:= \left( e_{l,m} \right)_{l,m=1}^d$ with $e_{l,m}= \sum_{k=0}^{n+d} \frac{z_l^k\,\overline{z_m}^{k}}{\omega_k},$ then there exists a constant
$\delta > 0,$ independent of $n,$ such that for every $n,$
\begin{equation}\label{DiagEst}
\det(E) \geq \delta \cdot (n+d+1)^{d+d_1} \cdot \prod_{l=1}^d |z_l|^{2(n+d+1)}. 
\end{equation}
\end{lem} 

The proof is similar to that of Lemma \ref{diagonal}, using instead the estimates from Lemma \ref{BergmanKernel}, and the details are left to the reader.

Lemma \ref{diagonalBergman} in turn allows us to obtain estimates on the decay of the coefficients $A_{i,n}$. Again, the proof is similar to the one for Lemma \ref{growthAin} and is omitted.

\begin{lem}\label{growthAinBergman}
Let $ \omega_k = \frac{1}{k+1}.$ Then the coefficients $A_{i,n}$ from Corollary \ref{main1} have the following decay, as $n \rightarrow \infty$:
\begin{equation*}
A_{i,n} = 
\begin{cases}
  O\left( \frac{1}{(n+d+1)^{2}} \right) \text{ for } 1 \leq i \leq d_1\\      \quad \\
  O\left( \frac{1}{(n+d+1)^2|z_i|^{n+d+1}} \right) \text{ for } d_1 < i \leq d.
\end{cases}
\end{equation*}
Consequently, for each $1 \leq i \leq d_1,$ there exists a constant $C_i$, independent of $n,$ such that 
$$ \sum_{k=0}^{n+d} \left| A_{i,n} \frac{\overline{z_i}^k}{\omega_k} \right| \leq C_i,$$ while for 
$d_1 < i \leq d,$ we have
$$ \sum_{k=0}^{n+d} \left| A_{i,n} \frac{\overline{z_i}^k}{\omega_k} \right| \rightarrow 0 $$ as $n \rightarrow \infty.$
\end{lem}

Using Lemma \ref{growthAinBergman}, the proofs of Theorems \ref{Wiener} and \ref{pointwise} for the Bergman space
$A^2$ now follow in the same manner as in Section \ref{Sec3}.

\section{General estimates for monotonic weights and roots on the unit circle}\label{new6sect}
In Sections \ref{Sec3} and \ref{Secnew5} we treated the cases of Hardy and Bergman spaces. The proofs there depend on the nature of the corresponding reproducing kernels and hence can not be extended directly to the general case. However, if we assume that the sequence of weights $\{\omega_{k}\}_{k\in\mathbb{N}}$ is monotonic and that $f$ is a polynomial with simple zeros \emph{on the unit circle}, we can derive more general estimates, which simplify the proofs of two of our main results. From here on, we assume that the weight $\omega$ defining the space $H^2_{\omega}$ satisfies that
\begin{equation}\label{eqn999}
\sum_{n=0}^{\infty} \frac{1}{\omega_n} = +\infty,
\end{equation}
and that there exists a constant $C >0$ such that for all $n \in \N$, and all $t \in \{0,...,n+1\}$
\begin{equation}\label{eqn998}
C^{-1} \omega_n \leq \omega_{n+t} \leq C \omega_n.
\end{equation}
The classical weights $\omega_k=(k+1)^{\alpha}$, for Dirichlet-type spaces do clearly have the first property if and only if $\alpha \leq 1$ while they always have the second one with $C=2^{|\alpha|}$. The first condition is used in \cite{FMS} and characterizes the cyclicity of $f(z)=1-z$ in $H^2_{\omega}$, while the doubling condition \eqref{eqn998} will be relevant in the following proof partly because it ensures that 
\begin{equation}\label{eqn997}
\sum_{k=0}^n \frac{\omega_n}{\omega_k} \rightarrow \infty, \quad {\mbox{ as } n} \rightarrow \infty.
\end{equation}

\begin{theorem}\label{Th2r2}
Let $f$ be a polynomial of degree $d$ with simple zeros such that $Z:=Z(f)\subset\mathbb{T}$ and let $K$ be a compact subset of $\overline{\mathbb{D}}\setminus Z$. Suppose $\omega=\{\omega_k\}$ is a monotone sequence that satisfies the conditions \eqref{eqn999} and \eqref{eqn998}. Then there is a positive constant $ C_1(Z,K)$, depending only on the zero set $Z$ and the set $K$, such that
\begin{equation}\label{r2}
\sup _{z \in K} |1-P_n(1)(z)| \leq C_1(Z,K) \cdot \left( \sum _{k=0}^{n}\frac{ s _{n}}{\omega _{k}}\right) ^{-1},
\end{equation}
where $s_n = \min \{1, \omega_n\}$.
In particular, if $p_n$ is the $n$-th optimal approximant to $1/f$ in $H_{\omega}^{2}$, then 
\begin{equation*}
1-p_n f\to 0 \ \ \  \mbox{ as } n\to \infty
\end{equation*}
uniformly on compact subsets of $\overline{\mathbb{D}}\setminus Z$. Moreover, there is a constant $C_2(Z,\omega)>0$ such that for all $n\in\mathbb{N}$,
\begin{equation}
\|1-p_n f\| _{A(\mathbb{T})}\leq  C_2(Z,\omega).
\end{equation}
\end{theorem}

\begin{rem}
In the above result, the value of $s_n$ in the estimate \eqref{r2} cannot be improved: at least when $\omega$ is non-decreasing, the value of $(1-p_nf)(0)$ decays at a comparable speed  to that of the right-hand side, so the estimate is exact for any compact subset of $\overline{\mathbb{D}}\setminus Z(f)$ containing the point 0. For the case of decreasing weights, we wonder whether the rate presented is sharp, since the same estimate for the value of $(1-p_nf)(0)$ holds but now a gap appears between the two quantities. \end{rem}

\begin{proof}
Let $Z=\{z_1, z_2, \dots, z_d\}$ and let $K$ be a compact subset of \linebreak $\overline{\mathbb{D}}\setminus Z$.
Recall that $k_n(z,w) = \sum_{k=0}^n \frac{\overline{w}^k z^k}{\omega_k}$.  In view of Theorem \ref{main} and Corollary \ref{estinv}, to prove (\ref{r2}), it suffices to show that there is a positive constant $ C_1(Z,K)$ such that, for each $n$, 
$$ \sup _{z\in K} |k_{n}(z, z_i)| \leq \dfrac{ C_1(Z,K)}{s _{n}}  \ \ \ (i=1,2,\dots, d).$$ 
Since $\omega$ is monotone, the sequence \[\psi(t) := \frac{1}{\omega_{t+1} } - \frac{1}{\omega_t}, \quad t  \in \N,\] has constant sign. In particular, we have that 
\begin{equation}\label{final101}\sum_{t=0}^n |\psi(t)| = \left|\sum_{t=0}^n \psi(t)\right| = \left|\frac{1}{\omega_{n+1}} -1\right|.
\end{equation}
On the other hand, we have that $\psi$ can be extended to a continuous function $\phi'$ with a monotone primitive $\phi$ such that $\phi(k)=\frac{1}{\omega_k}$ for all $k\in \N$. Abel's summation formula then gives that 
\begin{equation}\label{final102}
\sum_{k=0}^n (\bar{z_i}z)^k \phi(k) = \left(\sum_{k=0}^n (\bar{z_i}z)^k\right) \cdot \frac{1}{\omega_n} - \int_0^n \left(\sum_{k=0}^{\left \lfloor{t}\right \rfloor}  (\bar{z_i}z)^k\right) \phi'(t)dt,
\end{equation} and this shows that for any $z \in K$ and $z_i \in \T$, we have
\begin{equation}\label{final103}
|k_n(z,z_i)|\leq \left(\sup_{t \in \N} \frac{|1-(\bar{z_i}z)^t|}{|1-(\bar{z_i}z)|} \right) \cdot \left( \frac{1}{\omega_n} + \frac{1}{\omega_{n+1}} +1 \right). 
\end{equation} Using \eqref{eqn998} we see that 
\[\left( \frac{1}{\omega_n} + \frac{1}{\omega_{n+1}} +1\right) \leq \frac{2+C}{s_n},\]
 while the first term on the right-hand side of \eqref{final103} is bounded by a constant that depends only on the choice of the compact set $K \subset \overline{\D}\backslash Z$.  This concludes the proof of \eqref{r2}.

Since the weight $\omega$ satisfies \eqref{eqn999} and \eqref{eqn998}, we conclude that the right-hand side of (\ref{r2}) tends to $0$ as $n\to\infty$. Hence $\sup _{z\in K}|1-p_n (z)f(z)|= \sup _{z \in K} |1-P_n(1)(z)|\to 0$ as $n\to\infty$.

To see the Wiener algebra norm estimate, notice that from Corollary \ref{main1}
\[\|1-p_nf\|_{A(\T)} = \sum_{k=0}^{n+d} \left| \frac{1}{\omega_k} \sum_{i=1}^{d}A_{i,n}\overline{z_i}^k\right|.\]

By Corollary \ref{estinv}, the right-hand side above is bounded above by
\[\sum_{k=0}^{n+d} \frac{1}{\omega_k} \frac{d \cdot C(Z,\omega)}{\sum_{k=0}^n \frac{1}{\omega_k}},\] where $C(Z,\omega)$ is a positive constant.
The doubling condition 
\eqref{eqn998} ensures now that the last quantity is bounded by a positive constant
$ C_2(Z,\omega)$.
\end{proof}

\section{Higher multiplicity}\label{Sec7}

In the previous sections we studied the case of functions which have only simple zeros. Now we focus on functions of the form $g_d(z)=(z-1)^d$ for any $d \in  \N$. Even though $g_d$ can be treated (for any fixed $d$) as the limiting case of a sequence of functions with simple zeros only, the separate study of this case will shed some light on how to eliminate the assumption of simple zeros in Theorems \ref{main}, \ref{Wiener} and \ref{pointwise}.  In previous work, the study of the approximants to 
$1/g_d$ had only been fruitful in the two simplest cases: $d=0,1$ or $H^2_\omega=H^2$ (see \cite{blSimanek}). We denote by $v^t$ the transpose of the vector $v$, and by $v_0$ a column vector of zeros. We will prove the following result:

\begin{thm}\label{thm7}
Let $d, n \in \N$, $g_d(z)=(z-1)^d$ and denote by $p_n$ the $n$-th optimal approximant to $1/g_d$ in $H^2_\omega$. Then there exists a vector of constants $A_n= (A_{1,n},...,A_{d,n})$ such that for all $k \in \N$ with $0 \leq k \leq n+d$, we have
\begin{equation}\label{multi1}
\widehat{(1-p_ng_d)}(k) = \frac{A_{1,n}+A_{2,n} k +...+A_{d,n}k^{d-1}}{\omega_k}=\frac{1}{\omega_k} (1,k,...,k^{d-1}) \cdot A_n^t.\end{equation} 
Moreover $A_n$ is the unique solution to the linear system
\begin{equation}\label{multi5}
E \cdot A_n^t = \begin{pmatrix} 1 \\ v_0 \end{pmatrix},\end{equation}
where $E_{i,j} = \displaystyle{\sum_{k=0}^{n+d} \frac{k^{i+j-2}}{\omega_k}}$ for $i, j= 1,...,d$ and $E=(E_{i,j})_{i,j=1,...,d}$.
In particular $E$ is invertible and
\begin{equation}\label{distform}
\|1-p_ng_d\|^2_\omega= {\dist}^2 (1, \mathcal{P}_n \cdot g_d) = (1 - p_n g_d)(0)= A_{1,n} =  E^{-1}_{1,1}.
\end{equation} 
\end{thm}

\begin{proof}

The orthogonality conditions $(1-p_ng_d) \perp z^kg_d$ for 
$k = 0, 1, \ldots, n$ give rise to recurrence relations that the coefficients 
$\omega_k \cdot \widehat{(1-p_ng_d)}(k)$ must satisfy.  These recurrence relations are well-known (see 
\cite[Section 2.1]{GreKnu}), and lead directly to the condition \eqref{multi1} stated in Theorem \ref{thm7}.

Now, since $g_d$ has a zero of multiplicity $d$ at $1$,  the derivative of order $s$ of $ 1-p_ng_d$ must satisfy:

\begin{equation}\label{mzero}(1-p_ng_d)^{(s)}(1) = \begin{cases}
  1 \text{ if } s=0 \\      
  0 \text{ if } s=1,...,d-1.
\end{cases}\end{equation}

On the other hand, using  \eqref{multi1}, we must also have: 

\begin{equation}\label{multi2}
(1-p_ng_d)^{(s)}(1) = \sum_{k=s}^{n+d} \widehat{(1-p_ng_d)}(k) \frac{k!}{(k-s)!} = \sum_{k=s}^{n+d} \frac{(1,k,...,k^{d-1}) \cdot A_n k!}{\omega_k (k-s)!}.
\end{equation}

For $s=0$, condition \eqref{mzero} combined with \eqref{multi2} gives 
\begin{equation}\label{firstcoeff}
\left( \sum_{k=0}^{n+d} \frac{1}{\omega_k}, \sum_{k=0}^{n+d} \frac{k}{\omega_k}
,..., \sum_{k=0}^{n+d} \frac{k^{d-1}}{\omega_k}
 \right) \cdot A_n = 1,
\end{equation}
while for $s=1$ we obtain
\begin{equation*}
\left( \sum_{k=0}^{n+d} \frac{k}{\omega_k}, \sum_{k=0}^{n+d} \frac{k^2}{\omega_k}
,..., \sum_{k=0}^{n+d} \frac{k^{d}}{\omega_k}
 \right) \cdot A_n = 0.
\end{equation*}
Note that in this second equality, all the terms for $k=0$ in the sums are equal to 0, and thus we can start the summation at $k = 0.$
Continuing in this manner, using induction and noting that the quotient $\frac{k!}{(k-s)!}$ vanishes for $k=0,...,s-1$, we conclude that:
\begin{equation}\label{multi3}
\left( \sum_{k=0}^{n+d} \frac{k^{i-1}}{\omega_k}, \sum_{k=0}^{n+d} \frac{k^i}{\omega_k}
,..., \sum_{k=0}^{n+d} \frac{k^{i+d-2}}{\omega_k}
 \right) \cdot A_n = 0,
\end{equation}
for $i=2,...,d$.
Putting  \eqref{firstcoeff} and \eqref{multi3} together we conclude that \eqref{multi5} holds.

The remaining point is to check that $\det(E) \neq 0$. To see this, notice that $E$ is a Gram matrix for the inner product in $H^2_{1/\omega}$ (the space with weights given by the inverse of each $\omega_k$). In that space, $E_{i,j}=\left< f_i , f_j \right>$ where the coefficients of the functions $f_i$ in the orthonormal basis of monomials in $H^2_{1/\omega}$ are given by: 
\begin{equation}\label{multi6}
(\widehat{f_i}(k))_{k\in \N}= (1,2^{i-1},3^{i-1},...,(n+d)^{i-1},0,0,...).
\end{equation}

As discussed in Section 2, a basic result in linear algebra yields that $\det(E) \neq 0$ if and only if $\{f_i\}$ is a linear independent family. Since $n \geq 0$, such independence will be established if the matrix 
\begin{eqnarray*}
V=\begin{pmatrix}
1 & 1 & \cdots & 1    & 1 \\
1 & 2 & \cdots & d-1 & d \\
\vdots & \ddots & \vdots & \vdots & \vdots \\
1 & 2^{d-1} & \cdots & (d-1)^{d-1} & d^{d-1}
\end{pmatrix}
\end{eqnarray*}
has nonzero determinant, but $V$ is the transpose of the Vandermonde matrix for the points $\alpha_i = i$, for $i=1,...,d$ and hence its determinant \[\det(V)= \prod_{\stackrel{ i,j=1}{i < j}}^{d} (i-j) \neq 0.\]  Thus, $E$ is invertible, and the last identity in \eqref{distform} follows from \eqref{multi5}. The first two identities in \eqref{distform} are direct consequences of the fact that $p_ng_d$ is the orthogonal projection of the function 1, while $(1-p_ng_d)(0)=A_{1,n}$ is the case $k=0$ in \eqref{multi1}.
\end{proof}

\begin{rem} Notice that $E$ is a Hankel matrix, that is, a square matrix whose skew-diagonals are constant, and, as already mentioned, a Gramian. Hankel Gramians are interesting and have connections with the Hamburger moment problem, which asks when a sequence of real numbers corresponds to the moments of a positive Borel measure on the real line.   It may also be interesting to explore Levinson or Schur algorithms for inversion of Toeplitz and Hankel matrices in this context.  For the time being, we are going to use the information we just obtained in the context of Dirichlet-type spaces where $\omega_k=(k+1)^{\alpha}$.
\end{rem}

We first consider the case $\alpha <1$, and deal later with the classical Dirichlet space ($\alpha =1$). Note that the behavior in terms of cyclicity for $\alpha > 1$ is well understood, and we do not extend our treatment there. We would like to find good estimates of $A_{1,n}$ for large $n$ and will see that a certain Hilbert matrix will play a role.
Finally, notice that the rate at which $A_{1,n}$ decays towards 0 (as $n \rightarrow \infty$) is known from \cite{BCLSS}, up to a constant, but here we will determine the exact term that dominates this rate (including the constant). Denote by $B(t,s)=\int_0^1 u^{t-1} (1-u)^{s-1} du$ the classical beta function.

\begin{thm}\label{thmfinal}
Let $d \in \N$ be fixed, $g_d(z)=(z-1)^d$, $\omega_k = (k+1)^\alpha$, for all $k \in \N$, $A_{1,n}$ as in Theorem \ref{thm7}, and $\alpha < 1$. As $n \rightarrow \infty$ we have
\[A_{1,n} = \frac{ n^{\alpha -1} ( 1 + o(1))}{(B(d,1-\alpha))^2(1-\alpha)} .\]
\end{thm}

\begin{proof}
From Theorem \ref{thm7} we know that $A_{1,n}=E^{-1}_{1,1}$. Let $p > -1$. A standard estimate gives that if $p \geq 0$
\begin{equation}\label{multi30}
1 + \int_1^{n+d} x^p dx \leq \sum_{k=1}^{n+d} k^p \leq 1 + \int_2^{n+d+1} x^p dx,\end{equation} while if $-1<p \leq 0$ the inequalities are reversed. In either case, we see that, as $n \rightarrow \infty$, 
\begin{equation}\label{multi8}
\sum_{k=1}^{n+d} k^p = \frac{(n+d)^{p+1}}{p+1} (1 + o(1)).
\end{equation}
By the linear properties of determinants related to multiplication of rows and columns by a scalar, we have that
\begin{equation}\label{multi10}
\det(E)= \prod_{i=1}^{d} (n+d)^{i-1-\alpha} \cdot \prod_{j=1}^d (n+d)^j \cdot \det(E^{2)}) = \det(E^{2)}) (n+d)^{d(d-\alpha)},
\end{equation}
where $E^{2)}$ is the matrix with entries given by \[E^{2)}_{i,j} = \frac{E_{i,j}}{(n+d)^{i+j-1-\alpha}}= \frac{1+o(1)}{i+j-1-\alpha},\] for $i,j=1,...,d$.

Since the inversion of matrices is continuous among non-singular complex matrices, the determinant of the matrix $E$ satisfies
\begin{equation}\label{multi9}
\det(E^{2)}) = \det(E^{3)}) (1+o(1)),
\end{equation}
as $n \rightarrow \infty$, where $E^{3)}$ is the matrix with entries $E^{3)}_{i,j} = \frac{1}{i+j-1-\alpha}$, for $i,j=1,...,d$. In fact, by taking $n$ large enough we can make any of the minors of $E^{3)}$ arbitrarily close to any of those of $E^{2)}$. Even though $E^{3)}$ is a very ill conditioned matrix, we are taking arbitrarily small perturbations and then minors converge to the corresponding values.

Notice that $E^{3)}$ is still a Hankel Gram matrix, so algorithms for its inversion are abundant. However, $E^{3)}$ has even more structure and is usually referred to as a \emph{generalized Hilbert matrix}, which is a moment matrix associated with certain orthogonal systems. In particular, it is a Cauchy matrix, that is, a matrix whose $i,j$-th entry is of the form $a_{ij} = \frac{1}{x_i + y_j}$ for given sequences $x_i$ and $y_j$, $x_i \neq - y_j$.  Exact formulas are known for its determinant and for its inverse matrix (see, e.g., \cite[pp. 512-515]{High}.) Applying the formula for the inverse for the Cauchy matrix given by $x_i=i-1-\alpha$, $y_j=j$, we obtain

\begin{equation}\label{multi20}
(E^{3)})_{1,1}^{-1}  = \frac{1}{(1-\alpha) \cdot (B(d,1-\alpha))^2 }.
\end{equation}

We now apply Cramer's rule to obtain $E^{-1}_{1,1}$, by computing the determinant  of the lower-right $(d-1)$-dimensional principal minor of $E$, say $\widetilde{E}$. In that way,
\begin{equation}\label{multi21}
E^{-1}_{1,1}= \frac{\det(\widetilde{E})}{\det(E)} = \frac{ 
\det\begin{pmatrix}
1  & v_0^t \\
v_0  & \widetilde{E}
\end{pmatrix}}{\det(E)},
\end{equation}
where $v_0$ again denotes a column vector of zeros.

Using the relationship between $E$ and the matrix $E^{3)}$, we obtain
\[
E^{-1}_{1,1}= \frac{ 
\det\begin{pmatrix}
(n+d)^{\alpha-1}  & v_0^t \\
v_0  & E^{4)}
\end{pmatrix}}{\det(E^{3)})} \cdot (1+o(1)) = \] \[ =  (n+d)^{\alpha-1} (1+o(1)) (E^{3)})^{-1}_{1,1},
\]
where $E^{4)}$ is the lower right $(d-1)$-principal minor of $E^{3)}$.

Putting the last expression together with the expression of $(E^{3)})^{-1}_{1,1}$ from \eqref{multi20} finishes the proof.
\end{proof}

For the Dirichlet space case ($\alpha = 1$), the same approach with slightly different growth estimates still works. If in \eqref{multi30} we allowed $p=-1$, we would obtain that, as $n \rightarrow \infty$,
\begin{equation}\label{multi40}
 \sum_{k=1}^{n+d} k^{-1} = \log (n+d+1) (1+o(1)).
\end{equation}
Taking $i=j=\alpha=1$, $\log(n+d+1)$ grows much faster than the constant $1 = (n+d)^{i+j-1-\alpha}$, and the determinant of $E$ is in this case adjusted by a factor of $\log(n+d+1)$.
The same method as in the proof of Theorem \ref{thmfinal} yields
\begin{equation}\label{multi51}
A_{1,n} = E^{-1}_{1,1} = \frac{1+o(1)}{\log (n+d+1)}.
\end{equation}
Finally, notice that for all $d \in \N$, $d \geq 1$,
\[\lim_{\alpha \rightarrow 1} (1-\alpha) B(d,1-\alpha) = 1.\]
Using \eqref{multi8} once more, we can conclude the following:

\begin{corollary}
For $\alpha \leq 1$, $d \in \N$, the limit
\[L_{\alpha, d} := \lim_{n \rightarrow \infty} A_{1,n} \cdot \left(\sum_{k=0}^{n+d} \frac{1}{\omega_k}\right)\]
exists, depends continuously on $\alpha$ and satisfies
\begin{equation*}
L_{\alpha, d} =   \begin{cases}
  \frac{1}{((1-\alpha)\cdot B(d,1-\alpha))^2} & \quad \text{ if } \alpha < 1, \\      
  1 & \quad \text{ if } \alpha = 1.
\end{cases}
\end{equation*}
\end{corollary}

As remarked earlier, the results of this section give rise to the precise constant in the rate of decay of 
$\|1-p_ng_d\|_{\omega}^2$ for weights $\omega_k= (k+1)^{\alpha},$ $\alpha \leq 1.$  These estimates also provide a starting point for examining uniform convergence of $(1-p_ng_d)$ on compact subsets of the closed unit disc, which we leave for future work.   

\section{Concluding remarks}\label{Sec6}

We would like to conclude with some remarks and directions for future research.
\begin{itemize}
\item[(A)] The estimates obtained in Lemma \ref{growthAin} are close to
optimal: the estimates on $d_{k,n}$ yield that
the square of the norm of $1-p_n f$ is bounded by a constant times
$(\sum_{t=0}^{n+d} \frac{1}{\omega_t})^{-1}$, which we know from previous
work to be also the exact rate for any polynomial function with at
least one zero on the boundary.
\item[(B)] There are analogues to Theorems \ref{Wiener} and \ref{pointwise}, at least, for the
case when all the zeros are inside the disc $\D$, but one should
rather consider $h-p_nf$, instead of the function $1-p_nf$, where
$f$ has a factorization as $f=gh$ and $h$ is the orthogonal
projection of $1$ onto $[f]$ ($h$ is a constant multiple of an inner
function). This relation between factorization and the orthogonal projection of 1 onto invariant subspaces of $H^2_\omega$ is explained in \cite{Remarks}. 
\item[(C)] It seems natural to expect that Wiener algebra functions or 
functions in other $H^2_{\omega}$ spaces 
will have a behavior similar to the one described here, perhaps
requiring that boundary zeros are not multiple.  In order for the proofs here to 
go through, one needs good estimates, for $|z| \geq 1,$ on sums of the form 
$\sum_{k=0}^N \frac{1}{\omega_k} z^k$ that do not depend only on the modulus of $z.$
\item[(D)] The approach discussed in this paper yields, in the forthcoming paper \cite{SecoTellez}, the possibility of proving results on cyclicity and the corresponding rates of approximation on large classes of (non-Hilbert) Banach spaces.
\item[(E)] Theorem \ref{main} has not been established for functions $g$ other than very particular polynomials of low degree ($\deg(g) \leq \deg(f)$). However, it is of general interest to understand approximation properties for any function in the invariant subspace generated by a function $f$, even if we have to restrict to the case when $f$ is itself a polynomial (as in the present article).
\item[(F)] From the Taylor coefficients $d_{k,n}$ of $1-p_nf$ obtained in Corollary \ref{main1} we can also obtain the coefficients of $p_n$ themselves: Notice that $p_n = 1/f-(1-p_nf)/f$. Denote by $b_k$ the Taylor coefficient of order $k$ of the function $1/f$, and by $c_{k,n}$ those of $p_n$. Then for all $0 \leq k \leq n \in \N$ we have \[c_{k,n} = b_k + \sum_{r=0}^k b_{k-r}d_{r,n}.\]
\item[(G)] We think that our ideas give all the necessary tools to solve completely the problem of finding $1-p_nf$ for any polynomial $f$ in any space $H^2_\omega$ but the task of writing down the formulas for the case when $f$ has different zeros with different multiplicities seems involved and we decided not to pursue that here. However, the conditions that one will need to impose to determine the coefficients are that the corresponding derivatives (up to multiplicity of the zeros minus 1) cancel at the selected points. The key fact for our proof in the simple zero case was the fact that the orthogonal complement of $\mathcal{P}_nf$ in $\mathcal{P}_{n+d}$ is spanned by the kernels at the zeros of $f$. With higher multiplicities, one will need to make use of the derivatives of the kernels also, but much of the mathematical work will be simplified because derivatives of kernels happen to be the kernels in other space closely related to $H^2_\omega$. Also, the assumption that $f(0) \neq 0$ can be avoided by changing the weight $\omega$ to another weight which is essentially $\omega$ shifted as many times as the multiplicity of the zero at the origin.
\end{itemize}

\noindent\textbf{Acknowledgements.} Myrto Manolaki thanks the Department of Mathematics and Statistics 
at the University of South Florida for support during work on this project. Daniel Seco acknowledges
financial support from the Spanish Ministry of Economy and
Competitiveness, through the ``Severo Ochoa Programme for Centers of
Excellence in R\&D'' (SEV-2015-0554) and through grant
MTM2016-77710-P. The authors are grateful to the referees for their careful reading of the article and their useful comments.


\begin{thebibliography}{99}

\bibitem{BCLSS} \textsc{B\'en\'eteau, C., Condori, A., Liaw,
C., Seco, D.,} and \textsc{Sola, A.}, Cyclicity in Dirichlet-type
spaces and extremal polynomials,  \emph{J. Anal. Math.} {\bf 126}
(2015) 259--286.

\bibitem{Remarks} \textsc{B\'en\'eteau, C., Fleeman, M., Khavinson,
D., Seco, D.,} and \textsc{Sola, A.}, Remarks on inner functions and
optimal approximants, \emph{Canad. Math. Bull.} {\bf 61} (2018) 704--716.

\bibitem{BMS2} \textsc{B\'en\'eteau, C., Ivrii, O., Manolaki, M.,} and \textsc{Seco, D.}, 
Simultaneous zero-free approximation and universal optimal polynomial approximants, 
arXiv:1811.04308.

\bibitem{BKLSS} \textsc{B\'en\'eteau, C., Khavinson, 
D., Liaw, C.,Seco, D.,} and \textsc{Sola, A.}, 
Orthogonal polynomials, reproducing kernels, and zeros of optimal approximants,  
\emph{J. London Math. Soc.} {\bf 94} (2016), no. 3, 726--746.

\bibitem{blSimanek} \textsc{B\'en\'eteau, C., Khavinson, D., Liaw,
C., Seco, D.,} and \textsc{Simanek, B.}, Zeros of optimal polynomial
approximants: Jacobi matrices and Jentzsch-type theorems, \emph{Rev. Mat. Iber.} {\bf 35} (2019), no. 2, 607--642.

\bibitem{BS} \textsc{Brown, L. and Shields, A.},  Cyclic vectors in the Dirichlet space,
\emph{Trans. Amer. Math. Soc.} {\bf 285} (1984), 269-304.

\bibitem{Ch} \textsc{Chui, C.}, Approximation by double least-squares inverses, \emph{J. Math. Anal. Appl.} 
{\bf 75} (1980), 149--163.

\bibitem{Dur} \textsc{Duren, P. L.},
\emph{Theory of $H^p$ spaces}, Academic Press, New York, 1970.

\bibitem{DuS} \textsc{Duren, P. L.} and \textsc{Schuster, A.},
\emph{Bergman spaces}, AMS, Providence, RI, 2004.

\bibitem{EFKMR} \textsc{El-Fallah, O., Kellay, K., Mashreghi, J.,} and \textsc{Ransford, T.}, \emph{A primer on
the Dirichlet space}, Cambridge Tracts in Math. {\bf 203}, Cambridge
University Press, 2014.

\bibitem{FMS} \textsc{Fricain, E., Mashreghi, J.}, and
\textsc{Seco, D.}, Cyclicity in Reproducing Kernel Hilbert Spaces of
analytic functions, \emph{Comput. Methods Funct. Theory} (2014)
Issue 14, 665-680.

\bibitem{Gar} \textsc{Garnett, J. B.},
\emph{Bounded analytic functions}, Academic Press Inc., 1981.

\bibitem{GreKnu} \textsc{Greene, D. H.} and \textsc{Knuth, D. E.}, \emph{Mathematics for the Analysis of Algorithms}, Modern Birkh\"auser Classics, 3$^{rd}$ edition, 2008.

\bibitem{High} \textsc{Higham, N. J.},
\emph{Accuracy and Stability of Numerical Algorithms}, 2$^{nd}$ edition, SIAM, 2002.

\bibitem{HKZ} \textsc{Hedenmalm, H., Korenblum, B.,} and \textsc{Zhu, K.},
\emph{Theory of Bergman spaces}, Springer, New York, 2000.

\bibitem{HJ} \textsc{Horn, R. A.} and \textsc{Johnson, C. R.},
\emph{Matrix Analysis}, Cambridge University Press, 1985.

\bibitem{SecoTellez} \textsc{Seco, D.} and \textsc{T\'ellez, R.}, Polynomial approach to cyclicity for weighted $\ell^p_A$ spaces, in preparation.


\end{thebibliography}
\end{document}